\documentclass[aop,preprint]{imsart}

\RequirePackage[OT1]{fontenc}
\RequirePackage{amsmath,amsthm,amssymb,mathrsfs}
\RequirePackage[numbers]{natbib}
\RequirePackage[colorlinks,citecolor=blue,urlcolor=blue,breaklinks=true]{hyperref}

\usepackage[letterpaper]{geometry}
\startlocaldefs
\newcommand{\bb}{\ensuremath{\mathbb B}}
\newcommand{\N}{\ensuremath{\mathbb N}}
\newcommand{\G}{\ensuremath{\mathscr G}}
\newcommand{\bp}[1]{\ensuremath{\mathbb P}_\mu \left( #1 \right)}
\renewcommand{\leqslant}{\leq}

\newtheorem{theorem}{Theorem}
\newtheorem{lemma}[theorem]{Lemma}

\newtheorem{utheorem}{\textrm{\textbf{Theorem}}}

\theoremstyle{definition}
\newtheorem{definition}[theorem]{Definition}

\endlocaldefs

\begin{document}

\begin{frontmatter}

\title{The Hoffmann--J{\o}rgensen inequality\\ in metric semigroups\thanksref{T1}}
\runtitle{The Hoffmann--J{\o}rgensen inequality in metric
semigroups}

\begin{aug}
\author{\fnms{Apoorva}~\snm{Khare}\ead[label=e1]{khare@iisc.ac.in}}
\and
\author{\fnms{Bala}~\snm{Rajaratnam}\ead
[label=e2]{brajaratnam01@gmail.com}}

\runauthor{A. Khare and B. Rajaratnam}

\affiliation{Stanford University, University of California, Davis, and
University of Sydney}

\address[A]{Department of Mathematics\\
Indian Institute of Science\\
Bangalore 560012\\
India\\
\printead{e1}
}
\address[B]{Department of Statistics\\
Stanford University\\
390 Serra Mall\\
Stanford, California 94305\\
and\\
University of California, Davis\\
One Shields Avenue\\
Davis, California 95616\\
USA\\
\printead{e2}}


\received{\smonth{5} \syear{2016}}
\revised{\smonth{10} \syear{2016}}
\doi{10.1214/16-AOP1160}
\volume{45}
\issue{6A}
\firstpage{4101}
\lastpage{4111}
\pubyear{2017,}

\thankstext{T1}{Supported in part by US Air Force Office of Scientific
Research Grant award FA9550-13-1-0043, US National Science Foundation
under Grant DMS-0906392, DMS-CMG 1025465, AGS-1003823, DMS-1106642,
DMS-CAREER-1352656, Defense Advanced Research Projects Agency DARPA YFA
N66001-111-4131, the UPS Foundation, and SMC-DBNKY.}
\end{aug}

\begin{abstract}
We prove a refinement of the inequality by Hoffmann--J{\o}rgensen that is
significant for three reasons.
First, our result improves on the state-of-the-art even for real-valued
random variables.
Second, the result unifies several versions in the Banach space
literature, including those by
Johnson and Schechtman
[\textit{Ann. Probab.} \textbf{17} (1989) 789--808],
Klass and Nowicki
[\textit{Ann. Probab.} \textbf{28} (2000) 851--862],
and Hitczenko and Montgomery-Smith
[\textit{Ann. Probab.} \textbf{29} (2001) 447--466].
Finally, we show that the Hoffmann--J{\o}rgensen inequality (including
our generalized version) holds not only in Banach spaces but more
generally, in a very primitive mathematical framework required to state
the inequality: a metric semigroup $\mathscr{G}$. This includes normed
linear spaces as well as all compact, discrete, or (connected) abelian
Lie groups.
\end{abstract}



\end{frontmatter}

\section{Introduction}

In this paper, our goal is to present a broad generalization of the
Hoffmann--J{\o}rgensen inequality (see Theorem~\ref{Thj}). This is a
classical result in the literature, which is widely used in bounding sums
of independent random variables, with several different versions proved
in the general setting of a separable Banach space (see
\cite{HM,JS,KN,LT}). We recall a ``first version'' from the literature.

\begin{theorem}[Ledoux and Talagrand, {\cite{LT}, Proposition
6.7]}]\label{Thjlt}
Suppose $\bb$ is a separable Banach space, and $(\Omega, \mathscr{A},
\mu)$ is a probability space with $X_1, \dots, X_n \in L^0(\Omega,\bb)$
independent random variables.
For $1 \leqslant j \leqslant n$, define $S_j := X_1 + \cdots + X_j$ and
$U_n := \max_{1 \leqslant j \leqslant n} \| S_j \|$. Then
\[
\bp{U_n > 3t+s} \leqslant \bp{U_n > t}^2 +
\bp{\max_{1 \leqslant j \leqslant n} \| X_j \| > s},
\qquad \forall s,t \in (0,\infty).
\]
\end{theorem}

This version incorporates results by Kahane \cite{Kah} and
Hoffmann--J{\o}rgensen \cite{HJ}.
(See also \cite{HM} for a detailed history of the inequality.)

Theorem~\ref{Thjlt} has seen subsequent generalizations by several
authors, including Johnson and Schechtman \cite{JS}, Klass and Nowicki
\cite{KN}, and Hitczenko and Montgomery-Smith \cite{HM}. This last
variant is now stated as the following.

\begin{theorem}[Hitczenko and Montgomery-Smith, {\cite{HM}}, Theorem
1]\label{Thm}
(Notation as in Theorem~\ref{Thjlt}.) For all $K \in \N$ and $s,t \in
(0,\infty)$,
\[
\bp{U_n > 2Kt + (K-1)s} \leqslant \frac{1}{K!} \left( \frac{\bp{U_n >
t}}{\bp{U_n \leqslant t}} \right)^K + \bp{\max_{1 \leqslant j \leqslant
n} \| X_j \| > s}.
\]
\end{theorem}

While isoperimetric methods provide more powerful techniques to work
with, the aforementioned manifestations of the Hoffmann--J{\o}rgensen
inequality for Banach spaces also have numerous consequences in
estimating the magnitude and behavior of the quantities $\|S_n\|$ and
$U_n$, as explained in \cite{HM,KN}, for instance.

We now present several motivations behind the present paper. First, our
main result in Theorem~\ref{Thj} provides an improvement on
Theorems~\ref{Thjlt} and~\ref{Thm} above.
Note, Theorem~\ref{Thm} has a variant via the order statistics of the
variables $Y_j := \| X_j \|$ (see \cite{HM}). Our result improves on this
strengthening as well.

Second, it is not clear if either of Theorems~\ref{Thjlt} or~\ref{Thm}
follows from the other, or if they are even logically related. Our result
(Theorem~\ref{Thj}) simultaneously unifies and significantly generalizes
both of these results.

A third motivation arises out of independent mathematical and applied
interest. Note that to state the above inequalities, one requires merely
the notions of a metric and a binary associative operation. Thus, a
question of interest is to ascertain whether the result holds in the more
general setting of a separable metric semigroup $\G$ (defined below).

In this paper, we provide a positive answer to the above question. Thus,
we show Theorem~\ref{Thj} in a very primitive mathematical setting
required to state the Hoffmann--J{\o}rgensen inequality.
Our motivations in so doing are both modern as well as traditional.
Classically, a cornerstone of twentieth-century probability theory has
been the systematic and rigorous development of the field, for random
variables taking values in Banach spaces. At the same time, general
results on Fourier analysis and Haar measure for compact abelian groups,
and the study of random variables with values in metric groups
\cite{Gre,Ru} motivate the need to develop results in the greatest
possible generality. The present paper lies squarely in this area.

Additionally, an increasing number of modern-day settings involve working
outside the traditional Banach space paradigm. Indeed, settings of
compact and abelian Lie groups are studied in the literature, including
permutation groups, lattices and other discrete (semi)groups, circle
groups and tori. Moreover, modern data are manifold-valued -- including
in real/complex Lie groups -- as opposed to the traditionally
well-studied normed linear spaces.
Other modern settings include the space of graphons with the cut-norm
\cite{Lo}, as well as the space of labelled graphs $\G(V)$ on a fixed
vertex set $V$, which was studied in \cite{KR1,KR2}. The space $\G(V)$
turns out to be a $2$-torsion group, and hence cannot embed as a subgroup
into a normed linear space. Thus, Banach space methods are not adequate
to study stochastic phenomena in modern-day settings. To this end, this
paper allows for studying tail estimates and bounding random sums in
greater generality.

\section{Metric semigroups and the main result}

We now set some notation and state our main result.

\begin{definition}
A \textit{metric semigroup} is defined to be a semigroup $(\G, \cdot)$
equipped with a metric $d_\G : \G \times \G \to [0,\infty)$ that is
translation-invariant:
\begin{equation}
d_\G(ac,bc) = d_\G(a,b) = d_\G(ca,cb), \qquad \forall a,b,c \in \G.
\end{equation}

Equivalently, $(\G,d_\G)$ is a metric space equipped with a associative
binary operation $\cdot$ such that $d_\G$ is translation-invariant.
\end{definition}

Metric (semi)groups are ubiquitous in probability theory. Examples
include Banach spaces such as function spaces, discrete semigroups
(including finite groups as well as labelled graph space $\G(V)$
\cite{KR1,KR2}), and all compact or (connected) abelian Lie groups, which
include the circle and tori (via, e.g., \cite[Theorem V.5.3]{St}). Among
other examples are amenable groups (see \cite[Proposition 4.12]{CSC} and
the discussion around it) and abelian Hausdorff metrizable topologically
complete groups \cite{Kl}.

\begin{definition}
Suppose $(\G, d_\G)$ is a separable metric semigroup, with Borel
$\sigma$-algebra $\mathscr{B}_\G$. Given integers $1 \leqslant j
\leqslant n$ and random variables $X_1, \dots, X_n : (\Omega,
\mathscr{A}, \mu) \to (\G, \mathscr{B}_\G)$, define
\begin{equation}
S_j(\omega) := X_1(\omega) \cdots X_j(\omega), \qquad
M_j(\omega) := \max_{1 \leqslant i \leqslant j} d_\G(z_0, z_0
X_i(\omega)),
\end{equation}

\noindent where $z_0 \in \G$ is arbitrary. (We show below, $M_j$ is
independent of $z_0 \in \G$.)
\end{definition}

We now state our main result, namely, the aforementioned generalization
of the Hoffmann--J{\o}rgensen inequality, for separable metric semigroups.

\begin{utheorem}\label{Thj}
Suppose $(\G, d_\G)$ is a separable metric semigroup, $z_0, z_1 \in \G$
are fixed, and $X_1, \dots$, $X_n \in L^0(\Omega,\G)$ are independent.
Also fix integers $k, n_1, \dots, n_k \in \N$ and nonnegative scalars
$t_1, \dots, t_k, s \in [0,\infty)$, and define
\begin{equation}
U_n := \max_{1 \leqslant j \leqslant n} d_\G(z_1, z_0 S_j), \ \
I_0 := \{ 1 \leqslant i \leqslant k : \bp{U_n \leqslant t_i}^{n_i -
\delta_{i1}} \leqslant \frac{1}{n_i!} \},
\end{equation}

\noindent where $\delta_{i1}$ denotes the Kronecker delta.
Now if $\sum_{i=1}^k n_i \leqslant n+1$, then
\begin{align}\label{Ehj}
&\ \bp{U_n > (2 n_1 - 1) t_1 + 2 \sum_{i=2}^k n_i t_i + \left(
\sum_{i=1}^k n_i - 1 \right) s}\\
\leqslant &\ \bp{U_n \leqslant t_1}^{{\bf 1}_{1 \notin I_0}} \prod_{i \in
I_0} \bp{U_n > t_i}^{n_i} \prod_{i \notin I_0} \frac{1}{n_i!} \left(
\frac{\bp{U_n > t_i}}{\bp{U_n \leqslant t_i}} \right)^{n_i}\notag\\
&\ + \bp{M_n > s}.\notag
\end{align}

More generally, define
\begin{align*}
&\ K := \sum_{i=1}^k n_i, \qquad Y_j := d_\G(z_0, z_0 X_j),\\
&\ Y_{(1)} := \min(Y_1, \dots, Y_n), \quad \dots, \quad
Y_{(n)} := \max(Y_1, \dots, Y_n),
\end{align*}

\noindent so that $Y_{(j)}$ are the order statistics of the $Y_j$. Then
the above inequality can be strengthened by replacing $\bp{M_n > s}$ by
\[
\bp{\sum_{j=n-K+2}^n Y_{(j)} > (K-1)s}.
\]
\end{utheorem}

Theorem~\ref{Thj} generalizes the original Hoffmann--J{\o}rgensen
inequality in many ways: mathematically it is a significant
generalization of Theorem~\ref{Thjlt} (which itself generalizes the
classical Hoffmann--J{\o}rgensen inequality for Euclidean, Hilbert, and
Banach spaces). To see this, set
\[
\G = \bb, \quad z_0 = z_1 = 0, \quad k = 2, \quad n_1 = n_2 = 1, \quad
t_1 = t_2 = t.
\]

Now Theorem~\ref{Thjlt} follows from Theorem~\ref{Thj} with
$I_0 = \{ 1, 2 \}$.

Moreover, Theorem~\ref{Thj} also generalizes \cite[Theorem 1]{HM}, that
is, Theorem~\ref{Thm} -- which has different bounds than
Theorem~\ref{Thjlt}. To see this, set
$\G = \bb, \ z_0 = z_1 = 0, \ k = 1, \ n_1 = K, \ t_1 = t$.
Now the first expression on the right-hand side of equation \eqref{Ehj}
can be rewritten as follows:
\begin{equation}\label{Ecases}
\prod_{i=1}^k \bp{U_n > t_i}^{n_i} \min \left( 1, \frac{1}{n_i!
\cdot \bp{U_n \leqslant t_i}^{n_i - \delta_{i1}} } \right).
\end{equation}

Thus with the above values, Theorem~\ref{Thm} follows from
Theorem~\ref{Thj}:
\begin{align*}
&\ \bp{U_n > 2Kt + (K-1)s}\\
\leqslant &\ \bp{U_n > (2K-1)t + (K-1)s}\\
\leqslant &\ \bp{M_n > s} + \bp{U_n > t}^K \min \left( 1, \frac{1}{K!
\cdot \bp{U_n \leqslant t}^{K-1}} \right)\\
\leqslant &\ \bp{M_n > s} + \frac{1}{K!} \left( \frac{\bp{U_n >
t}}{\bp{U_n \leqslant t}} \right)^K.
\end{align*}

Second, in \cite{HM} it is not shown whether or not the variant of the
Hoffmann--J{\o}rgensen inequality (Theorem~\ref{Thm}) can be reconciled
with Theorem~\ref{Thjlt}. Our result achieves this goal, thus unifying
and simultaneously generalizing variants from the literature, including
by Johnson and Schechtman \cite{JS}, Klass and Nowicki \cite{KN} and
Hitczenko and Montgomery-Smith \cite{HM}.

Finally, Theorem~\ref{Thj} does not require a norm, group structure,
commutativity, or completeness, but is valid in the primitive
mathematical setting of separable metric semigroups. Thus, the result is
a significant generalization of the original inequality by
Hoffmann--J{\o}rgensen.

\section{Proof of the theorem}

In order to prove Theorem~\ref{Thj}, we first study basic properties of
metric semigroups $\G$. We begin with the \textit{triangle inequality} in
$\G$, which is straightforward, and used without further reference.

\begin{equation}
d_\G(y_1 y_2, z_1 z_2) \leqslant d_\G(y_1, z_1) + d_\G(y_2, z_2), \qquad
\forall y_i, z_i \in \G.
\end{equation}

We also require the following lemma, which provides a work around for the
``norm'' in a metric semigroup, when there is no identity element.

\begin{lemma}
Given a metric semigroup $(\G, d_\G)$, and $a,b \in \G$,
\begin{equation}\label{Estrict}
d_\G(a,ba) = d_\G(b,b^2) = d_\G(a,ab)
\end{equation}

\noindent is independent of $a \in \G$.
\end{lemma}

\begin{proof}
Compute using the translation-invariance of $d_\G$:
\[
d_\G(a,ba) = d_\G(ba, b^2 a) = d_\G(b, b^2) = d_\G(ab, a b^2) = d_\G(a,
ab). \qedhere
\]
\end{proof}

Now we show the main result of the paper.

\begin{proof}[Proof of Theorem~\ref{Thj}]
Our proof follows in part the argument in \cite{HM}; however, we are able
to streamline some of the steps and provide novel techniques that help
generalize the result to its present form. For convenience, the proof is
divided into steps.\smallskip

\textit{Step 1.}
Define $K := \sum_{i=1}^k n_i$, and given $1 \leqslant l \leqslant K$,
let $t'_l := t_i$ if $\sum_{j=1}^{i-1} n_j < l \leqslant \sum_{j=1}^i
n_j$. Also define
\begin{equation}
\zeta := (2 n_1 - 1) t_1 + 2 \sum_{i=2}^k n_i t_i + \left( \sum_{i=1}^k
n_i - 1 \right)s, \quad Y := \sum_{j=n-K+2}^n Y_{(j)}.
\end{equation}

Now if $Y > (K-1)s$, then it is clear that $M_n > s$. Thus, the
inequality is strengthened by replacing $\bp{M_n > s}$ by $\bp{Y >
(K-1)s}$. (Note that this strengthening of the inequality was originally
suggested in the setting of Banach spaces by Rudelson in \cite{HM}.) Now
set $\Omega_1 := \{ \omega \in \Omega : U_n(\omega) > \zeta, Y(\omega)
\leqslant (K-1)s \}$. Then
\[
\bp{U_n > \zeta} \leqslant \bp{Y > (K-1)s} + \bp{\Omega_1}.
\]

Thus, we will restrict ourselves to $\Omega_1$. Define $m_0 = m_0(\omega)
:= 0$, and let $m_1(\omega) > 0$ be the smallest integer such that
$d_\G(z_1, z_0 S_{m_1}(\omega)) > t_1$. Note that such an $m_1(\omega)$
exists because $t_1 \leqslant \zeta$ and $\omega \in \Omega_1$.\smallskip

\textit{Step 2.}
For this step, fix $\omega \in \Omega_1$. In this step, we inductively
define integers
\[
m_l = m_l(\omega),\ \text{ with }\
0 = m_0 < m_1 < m_2 < \cdots < m_K \leqslant n
\]

\noindent as follows: $m_1$ is as above, and given $m_{l-1}$ for $l>1$,
define $m_l$ to be the least integer $> m_{l-1}$ such that
$d_\G(S_{m_{l-1}}, S_{m_l}) > 2t'_l$. To do so, we first \textit{claim}
that such an integer $m_l$ exists for all $1 \leqslant l \leqslant K$.

To show this claim, suppose to the contrary that such an $m_l$ does not
exist (for the smallest such $l>1$). Then for all $\beta > m_{l-1}$,
$d_\G(S_{m_{l-1}}, S_\beta) \leqslant 2 t'_l$. We now make the
\textit{sub-claim} that
\[
d_\G(z_1, z_0 S_\alpha(\omega)) \leqslant t'_1 + \sum_{j=2}^l 2 t'_j +
(K-1) s \leqslant \zeta, \qquad \forall 1 \leqslant \alpha \leqslant n.
\]

Notice that the sub-claim contradicts the fact that we are restricted to
$\omega \in \Omega_1$, thereby proving the claim. Thus, it suffices to
show the sub-claim. To do so, we consider various cases: if $\alpha <
m_1$, then $d_\G(z_1, z_0 S_\alpha) \leqslant t'_1$, so we are done.
Next, if $\alpha \in (m_i, m_{i+1})$ for some $0<i<l-1$, then compute
using equation~\eqref{Estrict}, and that $Y \leqslant (K-1)s$ on
$\Omega_1$:
\begin{alignat*}{2}
&& d_\G(z_1, z_0 S_\alpha) \leqslant &\ d_\G(z_1, z_0 S_{m_1 - 1})
+ d_\G(z_0 S_{m_i-1}, z_0 S_{m_i}) + d_\G(z_0 S_{m_i}, z_0 S_\alpha)\\
&& {} + &\ \sum_{j=2}^i [d_\G(z_0 S_{m_{j-1} - 1}, z_0 S_{m_{j-1}}) +
d_\G(z_0 S_{m_{j-1}}, z_0 S_{m_j - 1})]\\
&& \leqslant &\ t_1 + \sum_{j=2}^i (Y_{m_{j-1}} + 2t'_j) + Y_{m_i} + 2
t'_{i+1} \leqslant t'_1 + 2 \sum_{j=2}^{l-1} t'_j + Y\\
&& \leqslant &\ t'_1 + 2 \sum_{j=2}^{l-1} t'_j + (K-1)s.
\end{alignat*}

There are two other cases with similar computations (hence are skipped):
\begin{itemize}
\item If $\alpha = m_i$ for some $i<l$, then
\[
d_\G(z_1, z_0 S_\alpha) \leqslant t'_1 + 2 \sum_{j=2}^i t'_j + (i+1)s
\leqslant t'_1 + 2 \sum_{j=2}^{l-1} t'_j + (K-1)s.
\]

\item If $\alpha \in (m_{l-1},n]$, then the sub-claim follows by using
that $d_\G(S_{m_{l-1}}, S_\alpha)$ $\leqslant 2 t'_l$ from above.
\end{itemize}

Proceeding by induction on $l$, the above analysis in this step proves
the claim about the existence of $0 = m_0(\omega) < \dots < m_K(\omega)
\leqslant n$, for all $\omega \in \Omega_1$.\medskip

\textit{Step 3.}
Given a strictly increasing sequence ${\bf m} := (m_1, \dots, m_K)$ such
that $0 = m_0 < m_1 < \cdots < m_K \leqslant n$, define $\Omega_{\bf m}$
to be the subset of all $\omega \in \Omega_1$ such that $m_i(\omega) =
m_i$ for all $i$. Then $\Omega_1$ is the disjoint union of the
$\Omega_{\bf m}$.

Now given $0 \leqslant \alpha < \beta \leqslant n$ and $t>0$, define:
\begin{align}
\begin{aligned}
p_{\alpha,\beta,t} := &\ \bp{d_\G(z_0 S_\alpha, z_0 S_\beta) > 2t
\geq d_\G(z_0 S_\alpha, z_0 S_j)\ \forall \alpha \leqslant j <
\beta},\\
\ & \\
p_{\beta,t} := &\ \bp{d_\G(z_1, z_0 S_\beta) > t \geq
d_\G(z_1, z_0 S_j)\ \forall 0 \leqslant j < \beta},
\end{aligned}
\end{align}

\noindent where by equation~\eqref{Estrict}, we may disregard the $z_0$'s
occurring in $p_{\alpha,\beta,t}$ except for $\alpha=0$, in which case we
define $z_0 S_0 := z_0$. Then by independence of the $X_j$ [and
equation~\eqref{Estrict}],
\[
\bp{\Omega_{\bf m}} \leqslant p_{m_1, t'_1} \prod_{j=2}^K p_{m_{j-1},
m_j, t'_j}.
\]

This allows us to continue the computations toward proving the result:
\begin{align}\label{Esum}
\begin{aligned}
\bp{U_n > \zeta} \leqslant &\ \bp{Y > (K-1)s} + \bp{\Omega_1}\\
\leqslant &\ \bp{Y > (K-1)s} + \sum_{{\bf m}} p_{m_1, t'_1} \prod_{j=2}^K
p_{m_{j-1}, m_j, t'_j}.
\end{aligned}
\end{align}

\textit{Step 4.}
For the next steps in the computations, we bound $\sum_{\beta = \alpha +
1}^\gamma p_{\alpha, \beta, t}$ in two different ways, where $\alpha,
\beta, \gamma \in \N$. First,
\begin{align*}
\sum_{\beta = \alpha + 1}^\gamma p_{\alpha,\beta,t} = &\
\bp{\max_{\beta \in (\alpha,\gamma]} d_\G(S_\alpha, S_\beta) > 2t}\\
\leqslant &\ \bp{\max_{\beta \in (\alpha,\gamma]} d_\G(z_1, z_0 S_\alpha) +
d_\G(z_1, z_0 S_\beta) > 2t}\\
\leqslant &\ \bp{2 U_\gamma > 2t} = \bp{U_\gamma > t}.
\end{align*}

(Here, $U_\gamma$ is defined similar to $U_n$.) Similarly,
\[
\sum_{\beta = 1}^\gamma p_{\beta,t} = \bp{\max_{\beta \in [1,\gamma]}
d_\G(z_1, z_0 S_\beta) > t} = \bp{U_\gamma > t}.
\]

Next, if $\bp{U_\alpha \leqslant t} > 0$, then using the independence of
the $X_j$,
\begin{align*}
& \sum_{\beta = \alpha + 1}^\gamma p_{\alpha,\beta,t}\\
&\ = \bp{\max_{\beta \in (\alpha,\gamma]} d_\G(S_\alpha, S_\beta) > 2t} =
\bp{\max_{\beta \in (\alpha,\gamma]} d_\G(S_\alpha, S_\beta) > 2t\ |\
U_\alpha \leqslant t}\\
&\ \leqslant \frac{\bp{\max_{\alpha < \beta \leqslant \gamma} d_\G(z_1,
z_0 S_\beta) > t \mbox{ and } \max_{1 \leqslant \beta \leqslant \alpha}
d_\G(z_1, z_0 S_\beta) \leqslant t}}{\bp{U_\alpha \leqslant t}}\\
&\ = \frac{1}{\bp{U_\alpha \leqslant t}} \sum_{\beta = \alpha + 1}^\gamma
p_{\beta,t}.
\end{align*}

These calculations are summarized in the following system of
inequalities:
\begin{align}\label{Ebounds}
\begin{aligned}
\sum_{\beta = \alpha + 1}^\gamma p_{\alpha, \beta, t} \leqslant
\bp{U_\gamma > t} = \sum_{\beta = 1}^\gamma p_{\beta,t},\\
\sum_{\beta = \alpha + 1}^\gamma p_{\alpha, \beta, t} \leqslant
\frac{1}{\bp{U_\alpha \leqslant t}} \sum_{\beta = \alpha + 1}^\gamma
p_{\beta,t}.
\end{aligned}
\end{align}

\textit{Step 5.}
We now perform what is in a sense the ``main step'' of the computation.
More precisely, we use the previous step to bound from above the
following expression from equation~\eqref{Esum}:
\[
\widetilde{S} := \sum_{{\bf m}} p_{m_1, t'_1} \prod_{j=2}^K p_{m_{j-1},
m_j, t'_j},
\]

\noindent where the summation is over all $0 < m_1 < \cdots < m_K
\leqslant n$.

For $1 \leqslant i \leqslant k+1$, define $s_i := \sum_{j=1}^{i-1} n_j$.
Then $t'_l = t_i$ for $s_i < l \leqslant s_i + n_i$. Suppose $k>1$. We
bound $\widetilde{S}$ via induction on $k$, presented here in a reverse
manner. Namely, we sum first over $m_j$ for $j \in (s_k, s_{k+1}] =
(K-n_k, K]$; then over $j \in (s_{k-1}, s_k]$; and so on, reducing to the
base case $k=1$ (addressed in the next step).
In the present step, we stop after one round of summation.
\[
\widetilde{S} = \sum_{{\bf m}_k} p_{m_1, t'_1} \prod_{j=2}^{s_k}
p_{m_{j-1}, m_j, t'_j} \cdot \sum_{\alpha_0 = m_{s_k} < \alpha_1 < \cdots
< \alpha_{n_k} \leqslant n} \prod_{j=1}^{n_k}
p_{\alpha_{j-1},\alpha_j,t_k},
\]

\noindent where the outer sum is over ${\bf m}_k := \{ m_j : j \leqslant
s_k \}$.
We claim that for all fixed $m_j$ for $j \not\in (s_k, s_k+n_k]$, the
inner sum can be bounded above by an expression occurring in
Theorem~\ref{Thj} [see~\eqref{Ecases}]. More precisely, we claim:
\begin{align}\label{Eestimate}
\begin{aligned}
&\ \sum_{\alpha_0 = m_{s_k} < \alpha_1 < \cdots < \alpha_{n_k} \leqslant
n} \prod_{j=1}^{n_k} p_{\alpha_{j-1},\alpha_j,t_k}\\
\leqslant &\ \bp{U_n > t_k}^{n_k} \min \left( 1, \frac{1}{n_k! \cdot
\bp{U_n \leqslant t_k}^{n_k}} \right)
\end{aligned}
\end{align}

\noindent (note, $k>1$). To see why, using~\eqref{Ebounds}, the sum
in~\eqref{Eestimate} is bounded above by
\begin{align*}
&\ \sum_{\alpha_0 = m_{s_k} < \alpha_1 < \cdots < \alpha_{n_k} \leqslant
n} \prod_{j=1}^{n_k} p_{\alpha_{j-1},\alpha_j,t_k}\\
= &\ \sum_{\alpha_0 = m_{s_k} < \alpha_1 < \cdots < \alpha_{n_k - 1}
\leqslant n} \prod_{j=1}^{n_k - 1} p_{\alpha_{j-1},\alpha_j,t_k} \cdot
\sum_{\alpha_{n_k} = \alpha_{n_k - 1} + 1}^n p_{\alpha_{n_k - 1},
\alpha_{n_k}, t_k}\\
\leqslant &\ \sum_{\alpha_0 = m_{s_k} < \alpha_1 < \cdots < \alpha_{n_k -
1} \leqslant n} \prod_{j=1}^{n_k - 1} p_{\alpha_{j-1},\alpha_j,t_k} \cdot
\bp{U_n > t_k}\\
\leqslant &\ \bp{U_n > t_k} \sum_{\alpha_0 = m_{s_k} <
\cdots < \alpha_{n_k - 2} \leqslant n} \prod_{j=1}^{n_k - 2}
p_{\alpha_{j-1},\alpha_j,t_k}\\
& \qquad \times \sum_{\alpha_{n_{k-1}} = \alpha_{n_{k-2} + 1}}^n
p_{\alpha_{n_{k-2}},\alpha_{n_{k-1}},t_k}.
\end{align*}

Continuing inductively, we obtain an upper bound of $\bp{U_n >
t_k}^{n_k}$.

Next, if $\bp{U_n \leqslant t_k} > 0$, then we bound the sum
in~\eqref{Eestimate} using~\eqref{Ebounds}, as in the proof of
\cite[Theorem 1]{HM}; this yields an upper bound of
\[
\sum_{\alpha_0 = 0 < \alpha_1 < \cdots < \alpha_{n_k} \leqslant n}
\prod_{j=1}^{n_k} p_{\alpha_{j-1},\alpha_j,t_k} \leqslant
\frac{1}{\bp{U_n \leqslant t_k}^{n_k}} \sum_{1 \leqslant \alpha_1 <
\cdots < \alpha_{n_k} \leqslant n} \prod_{j=1}^{n_k} p_{\alpha_j,t_k}.
\]

Since $n_k$ distinct numbers may be arranged in $n_k!$ ways, adopting an
argument in the proof of \cite[Theorem 1]{HM} shows the right-hand side
is at most
\begin{equation}\label{Elast}
\frac{1}{n_k!} \frac{1}{\bp{U_n \leqslant t_k}^{n_k}} \left( \sum_{\beta
= 1}^n p_{\beta,t_k} \right)^{n_k} = \frac{1}{n_k!} \frac{\bp{U_n >
t_k}^{n_k}}{\bp{U_n \leqslant t_k}^{n_k}}.
\end{equation}

This analysis proves the claim in~\eqref{Eestimate}. Note as
in~\eqref{Ecases}, the minimum corresponds precisely to whether or not $k
\in I_0$, as in the statement of the theorem. (The statement of the
result also includes the case when $\bp{U_n \leqslant t_k} = 0$.)\medskip

\textit{Step 6.}
Starting from~\eqref{Esum}, we now have a nested sum over $m_j$, $j \in
[1, s_k]$, as the estimate obtained in~\eqref{Elast} can be taken outside
the sum over the $m_j$. Repeat the computation in Step 5, summing over
the $m_j$ with $j \in (s_{k-1}, s_k]$; then over $j \in (s_{k-2},
s_{k-1}]$; and so on. This yields the expression for $k=1$:
\begin{align*}
\widetilde{S} \leqslant \prod_{1 < i \in I_0} & \bp{U_n > t_i}^{n_i}
\prod_{1 < i \notin I_0} \frac{1}{n_i!} \left( \frac{\bp{U_n >
t_i}}{\bp{U_n \leqslant t_i}} \right)^{n_i}\\
& \times \sum_{\{ m_j : j \in [1,n_1] \}} p_{m_1, t_1} \prod_{j=2}^{n_1}
p_{m_{j-1}, m_j, t_1}.
\end{align*}

It remains to find an upper bound for this last summation. To do so,
follow the computations in the previous step, using
equation~\eqref{Ebounds}. Thus, on the one hand, this summation is again
at most $\bp{U_n > t_1}^{n_1}$. On the other hand, it is bounded above,
assuming that $\bp{U_n \leqslant t_1} > 0$, by
\begin{align*}
&\ \frac{1}{\bp{U_n \leqslant t_1}^{n_1 - 1}} \sum_{1 \leqslant m_1 <
\cdots < m_{n_1} \leqslant n} \prod_{j=1}^{n_1} p_{m_j, t_1}\\
\leqslant &\ \frac{1}{\bp{U_n \leqslant t_1}^{n_1 - 1}} \frac{1}{n_1!}
\left(\sum_{\beta=1}^n p_{\beta,t_1} \right)^{n_1}\\
= &\ \bp{U_n > t_1}^{n_1} \cdot \frac{1}{n_1! \cdot \bp{U_n \leqslant
t_1}^{n_1 - 1}},
\end{align*}

\noindent and by equation \eqref{Ecases}, this completes the proof of the
theorem.
\end{proof}

\subsection*{Concluding remarks}

The validity of the Hoffmann--J{\o}rgensen inequality in the metric
semigroup setting suggests further work along two directions.
First, the Banach space version of this inequality is an important result
in the literature that is widely used in bounding sums of independent
Banach space-valued random variables. Having proved Theorem~\ref{Thj}, we
apply it in related work \cite{KR3} to obtain similar tail bounds for
sums of independent metric semigroup-valued random variables.
Additionally, in \cite{KR4} 
we study other probability inequalities for metric (semi)groups, such as
the Khinchin--Kahane inequality, together with its connections to
embedding abelian normed metric groups into (minimal) Banach spaces.

\section*{Acknowledgments}
We thank David Montague and Doug Sparks for carefully going through an
early draft of the paper and providing detailed feedback, which improved
the exposition. We also thank the anonymous referee for providing
technical feedback that helped clarify the proof.



\end{document}